\documentclass[12pt,a4paper]{article}
\usepackage{amsmath}
\usepackage{amsthm}
\usepackage{amssymb}
\usepackage{verbatim}
\usepackage{enumerate,color}
\usepackage{caption}
\usepackage[normalem]{ulem}
\usepackage{fancyheadings, geometry} % lo metto
\usepackage{kbordermatrix}
\usepackage{soul}
\usepackage{thmtools}

\usepackage{tikz}
\tikzstyle{u}=[draw=red,thick,circle,inner sep=4pt]

\newcommand{\N}{{\mathbb{N}}}
\newcommand{\dun}{\frac{1}{n}}

\newtheorem{thm}{Theorem}[section]
\newtheorem{defn}[thm]{Definition}
\newtheorem{prop}[thm]{Proposition}
\newtheorem{cor}[thm]{Corollary}
\newtheorem{lem}[thm]{Lemma}
\newtheorem{xmpl}[thm]{Example}
\newtheorem{rem}[thm]{Remark}

\title{The Shapley Value in the Knaster Gain Game}
\author{Federica Briata\footnote{University of Genova, Department of Mathematics, via Dodecaneso 35, 16146 Genova, Italy. 
Email: federica.briata@libero.it, 
Phone: +39(010)3536923} ,  
Andrea Dall'Aglio\footnote{Sapienza University of Rome, Department of Mathematics, Piazza Aldo Moro 5, 00185, Roma, Italy. 
Email: dallaglio@mat.uniroma1.it,
Phone: +39(06)49913248}, 
Marco Dall'Aglio\footnote{Corresponding author - LUISS University, Department of Economics and Finance, viale Romania, 32, 00197 Roma, Italy. 
Email: mdallaglio@luiss.it,
Phone: +39(06)85225639}, 
Vito Fragnelli\footnote{University of Eastern Piedmont - Department of Sciences and Innovative Technologies, Viale T.Michel, 11 - 15121 Alessandria, Italy.
Email: vito.fragnelli@uniupo.it,
Phone: +39(0131)360224}
}
\date{March 23, 2017}
\theoremstyle{plain}

\begin{document}

\renewcommand\thmcontinues[1]{Continued}

\maketitle

\begin{abstract}
In Briata, Dall'Aglio and Fragnelli (2012), the authors introduce a cooperative game with transferable utility for allocating the gain of a collusion among completely risk-averse agents involved in the fair division procedure introduced by Knaster (1946). In this paper we analyze the Shapley value (Shapley, 1953) of the game and propose its use as a measure of the players' attitude towards collusion. 
Furthermore, we relate the sign of the Shapley value with the ranking order of the players' evaluation, and show that some players in a given ranking will always deter collusion. Finally, we characterize the coalitions that maximize the gain from collusion, and suggest an ad-hoc coalition formation mechanism.

\end{abstract}


\section{Introduction}
\label{intro}

In Fragnelli and Marina (2009) the problem of manipulation and collusion in the Knaster procedure (1946) for completely risk-averse agents was posed and analyzed; later, Briata, Dall'Aglio and Fragnelli (2012) introduced a cooperative game with transferable utility, the so-called {\em gain game}, for allocating the gain of a collusion among the agents involved. In this paper we devote our attention to the computation of the Shapley value (Shapley, 1953) of the gain game, and we give a novel interpretation for it as an index of the colluding power of each agent.

The collusion is a secret and fraudulent agreement among two or more agents for an illicit purpose, at damage of other ones. Auctions and sports are also not immune from collusion. For example, auctions with low minimum prices are vulnerable to collusion among bidders. Graham and Marshall (1985) study the optimal minimum price set by a seller, while Mead (1967) and Milgrom (1987) prove that ascending-bid auction is more susceptible to collusion than sealed-bid auction. The possible agreements of agents in an auction and the consequent allocation of gains is considered by Branzei et al. (2009) and Fragnelli and Meca (2010). To avoid that judges of artistic sports collude, Federations adopt various strategies in the regulations (Gambarelli et al., 2012). We recall that the collusive behavior is illegal. For instance, the Italian Civil Code punishes the suspected or supposed colluders. 

Fair division procedures are certainly not immune from collusion. In what follows we focus our attention on the role that collusion plays into one of the earliest procedures proposed by one of the founders of the field. The Knaster procedure (1946) allocates indivisible objects with monetary compensations in order to restore fairness: each indivisible item is first exchanged for a money amount equal to the highest valuation of it, then the monetary quantity is shared among all the agents according to their valuations. Knaster procedure is \textit{efficient} (there is no other distribution that yields every agent a higher payoff), and \textit{proportional} (each of the $n$ agents thinks to receive at least one $n$-th of the total value), if the agents report their true valuations (Brams and Taylor, 1996 and 1999). When an agent misreports her/his valuation individually, Knaster procedure is manipulable, incurring the risk of a loss in her/his final payoff; on the other hand, Knaster procedure with infinitely risk-averse agents is non-manipulable, since there is no way of obtaining a \textit{safe} gain (Fragnelli and Marina, 2009). Nevertheless, if two or more agents (but not all) collude, coordinating their false declarations, Knaster procedure proves to be coalition-manipulable, where a mechanism is said \textit{coalition-strategy-proof} when ''{\em if a joint misreport by a coalition strictly benefits one member of the coalition, it must strictly hurt at least one other member}'' (Moulin, 1993). Fragnelli and Marina (2009) remark that the gain produced by the collusion is always non-negative, but enlarging the set of colluders the gain may increase or decrease. Briata, Dall'Aglio and Fragnelli (2012) propose a dynamic allocation mechanism according to which the enlargement of the set of colluders is always non-disadvantageous, since the previous gain of the incumbent colluders is guaranteed by their altered declarations, so they secure themselves against the entrant colluders.

The paper is organized as follows. In the next section, we recall the basic definitions of game theory, the Knaster procedure, the concept of collusion and the definition of gain game; in Section \ref{computation}, we develop a closed form formula for computing the Shapley value of the game in  polynomial time; Section \ref{attitude} is devoted to the analysis of the players' attitude towards collusion, as measured by the Shapley value; in Section \ref{maximal}, we characterize the coalitions that provide maximal total and per-capita gain, and propose an ad-hoc coalition formation mechanism based on this principle; Section \ref{concl} concludes.

%--------------------  --  -----------------

\section{Notation and Basic Definitions}
\label{basic}
In this section, we provide the basic notion of cooperative game theory and a short outline of the Knaster procedure, in the case in which a single object has to be assigned.
%------------------  --  ----------------------------
\subsection{Elements of Game Theory}
\label{subsect game theory}

A \textit{Transferable Utility game} or \textit{TU-game} in characteristic function form is a pair $(N,v )$, where $N=\{1,...,n\}$ is a finite \textit{set of players} and  $v:2^N \rightarrow \mathbb{R}$ is a real function, with $v(\varnothing)= 0$, called \textit{characteristic function}.
A subset $S \subseteq N$ is called \textit{coalition} and $N$ is called \textit{grand coalition}.  $v(S), S \subseteq N$, is the \textit{worth} of $S$, i.e. the utility that the players in $S$ may obtain independently from the other players.  A TU-game $( N,v )$ is \textit{inessential}, if $v(N) = \sum_{i \in N} v(\{i\})$. Given a TU-game $( N,v )$, an \textit{allocation} is a vector $(x_i)_{i \in N}\in\mathbb{R}^n$; an \textit{imputation} is an allocation such that $\sum_{i \in N}x_i=v(N)$ (\textit{efficiency}) and $x_i \ge v(\{i\})$ for each $i \in N$ (\textit{individual rationality}); an \textit{allocation rule} is a function $\psi:(N,v) \rightarrow \mathbb{R}^n$ which assigns an allocation $\psi(v)$ to every TU-game $( N,v )$ in the class of games with player set $N$. One of the most usual rules is the \textit{Shapley value} (Shapley, 1953), $\phi$, given by:
\begin{equation} 
\label{shapley}
\phi_i (v)= \sum_{S \subseteq N \setminus \{i\}}\frac{s! (n-s-1)!}{n!}(v(S\cup \{i\})-v(S)), i\in N,
\end{equation}
where $s=|S|$, the cardinality of $S$. An alternative equivalent formula for the Shapley value (see Hart, 1989) is
\[
\phi_i(v)=\frac{1}{n!} \sum_{\pi \in \Pi} \left[ v(P_i^\pi \cup \{i\})-v(P_i^\pi) \right] \; ,
\]
where $\Pi$ is the set of all the permutations of the elements of $N$ and $P_i^\pi$ is the set of players in $N$ which precede $i$ in the order $\pi$. In words, the Shapley value of  a player is her/his expected marginal contribution to a random coalition.

%------------------  --  ----------------------------
\subsection{Knaster procedure for one object}
\label{subsect Knaster one object}

Applying the Knaster procedure, we suppose that the value of each object obtained by an agent is independent from who has obtained the other objects ({\em additivity}), so the problem of allocating a set of objects simply corresponds to treating each object independently (private communication by Fink to Brams, mentioned in Brams and Taylor, 1996); this enables us to consider only a single object. 

Let $N=\{1, \ldots, n \}$ be the set of agents, which we assume to be completely risk-averse, to have the same valuation of monetary quantities, and to have equal rights on the object. We suppose that agent $i \in N$ knows only her/his own \textit{valuation} $v_{i}$ of the item and does not use any statistical information on the valuations of the others. We assume also that agents are not subject to any liquidity or budget constraints. Without loss of generality, we assume that the agents are ordered according to weakly decreasing valuations, i.e. $v_1 \ge v_2 \ge \ldots \ge v_n$. Once the valuations are communicated to a mediator, agent 1 gets the object for the price $v_1$. Notice that even in the case of multiple maximal valuations, the transaction will involve only one agent, labelled agent 1.
Exchanging the indivisible item for money makes the division possible. Each agent $i \in N$ receives the expected \textit{initial fair share} $E_i=\frac{1}{n}v_{i}$, plus an equal share of the \textit{surplus} ${\cal S}=v_1-\frac{1}{n}\sum_{j\in N}v_{j}.$ The surplus is non-negative (Brams and Taylor, 1996); in particular it is zero if and only if all the valuations are exactly the same (Kuhn, 1967). The \textit{adjusted fair share} or payoff of agent $i\in N$  is $V_i= E_i + \frac{{\cal S}}{n}= \frac{v_i}{n}+ \frac{v_1}{n} - \frac{1}{n^2}\sum_{j\in N}v_{j}$. 

In the resulting allocation, the sum of the compensations in money $c_1= \frac{1-n}{n}v_1+ \frac{{\cal S}}{n} = V_1 - v_1$ and $c_i= \frac{v_i}{n} + \frac{{\cal S}}{n}= V_i$, for $i\in N \setminus \{1\}$ is zero, so Knaster procedure does not require or produce money (Brams and Taylor, 1996). The solution is proportional since it secures to agent $i \in N$ a portion $V_i \ge \frac{1}{n} v_i$. Knaster procedure with more than two agents does not guarantee envy-freeness, as an agent may prefer another's portion to her/his own. For instance, agent $k$ envies agent $j$, with $1<j<k  \le n$, when $v_k < v_j$, as in this case $c_k < c_j$.

%------------------  --  ----------------------------
\subsection{Collusive behavior and gain game}
\label{subsect collusion}

From  Fragnelli and Marina (2009), when a single agent $k \in N \setminus \{1\}$  misrepresents her/his valuation $v_k$ declaring $v_k + \varepsilon$,  with $0 < \varepsilon < v_1-v_k$, and the others maintain their valuations $v_i$ for each  $i \in N \setminus \{k\}$, we have the following consequences:
\begin{enumerate}
\item  agent 1 still has the highest declaration, gets the object, pays $\frac{n-1}{n} v_1$, receives $\dun \left(v_1 - \dun\sum_{j \in N} v_j - \dun\varepsilon\right)$ and her/his payoff decreases by $\frac{1}{n^2}\varepsilon$;
\item the final amount for agent $k$ is $\dun v_k + \dun\varepsilon + \dun v_1 -  {\frac{1}{n^2}}\sum_{j \in N} v_j  - \frac{1}{n^2}\varepsilon$, so the variation of the payoff is $\frac{n-1}{n^2}\varepsilon$;
\item each agent $i \in N \setminus \{1,k\}$ receives $\dun v_i + \dun \left(v_1 -  \dun\sum_{j \in N} v_j -  \dun\varepsilon\right)$ and her/his payoff decreases by $\frac{1}{n^2}\varepsilon$.
\end{enumerate}

Then, we recall Definition 1 in Fragnelli and Marina (2009).
\begin{defn}
A collusion of a coalition of completely risk-averse agents consists of:
\begin{enumerate}
\item truthful revelation among them of their valuations;
\item same declaration of the highest true valuation;
\item binding agreement on the gain sharing.
\end{enumerate}

\end{defn}
Joining Propositions 1 and 2 in Fragnelli and Marina (2009), and generalizing them to any number of agents, we get the following result.

\begin{prop}
\label{gaingame}
The {\em highest safe gain} of a set of agents $S \subseteq N$ of completely risk-averse agents is obtained when they truthfully reveal their valuations to the other colluders and all of them declare the same value $b^S$ where
$$b^S:=\max_{i \in S} v_i \; . $$
The corresponding joint gain of the set of agents $S$ is
$$v_g(S) = \dfrac{n-s}{n^2} \sum_{i \in S} (b^S - v_i) \; .$$
\end{prop}

In other words, the collusion starts with a truthful revelation among the colluding agents of their valuations, then all of them declare the highest true valuation and a binding agreement guarantees their safe gain and its sharing. Clearly $v_g(S) \geq 0$ for every $S \subseteq N$. 
Interpreting $v_g$ as a characteristic function, we obtain the {\em gain game} $(N,v_g)$.
We remark that the gain game introduced in Proposition \ref{gaingame} is inspired by the collusion game defined in Briata, Dall'Aglio and Fragnelli (2012); the main difference is that a collusion game may be defined with each group of colluders as player set, while the player set of the gain game is the
grand coalition.

Furthermore, $v_g(\{i\})=0$ for every $i \in N$ and $v_g(N)=0$, so the game is inessential. Consequently, the non-null game has no monotonicity property, either in absolute or in relative value, and the contribution of a player to a coalition can be positive, null or negative.
In this situation the Shapley value is not an imputation and cannot be interpreted as the optimal division of the grand coalition payoff, since such coalition will not form. We will rather resort to the alternative interpretation of the Shapley value of a player as his/her expected marginal contribution to a random coalition; Consequently, it may be viewed as an index of the colluding power of each agent.
Recalling the efficiency of the Shapley value, $\sum_{i \in N}\phi_i(v_g)=v_g(N)=0$, we can classify players into three classes, depending on the sign of the corresponding Shapley value: Those favoring collusion (positive value), those neutral w.r.t.\ it (null value), and those inhibiting it (negative value). Note that the matter of being favorable, neutral or adverse towards collusion has no ethical meaning; it simply refers to the expected gain after collusion. For instance, they may be less or more interested in implementing mechanisms that aim to make collusions more difficult. A theoretical application is an ex-ante analysis of the profitability for an agent to participate in a colluding group, supposing we know the ex-post valuations, like an impartial external observer with complete information. The Shapley value may be viewed also as an insurance against collusions, that the agents with negative values pay to the agents with positive values. Of course, from a practical point of view, it is difficult to have such complete information.

\section{Computing the Shapley Value}
\label{computation}

In this section, we present the results that allow for a quantitative analysis of the Shapley value for the gain game defined in the previous section. 
Our main result gives an explicit expression for the Shapley value of the gain game as linear combination of the differences between adjacent players.
\begin{thm}
\label{linshap_expl}
For each agent $i \in N$,
\begin{equation}
\label{linshap_eq}
\phi_i(v_g)=\sum_{j =1}^{n-1} \psi_{ij}( v_{j} - v_{j+1}) \; ,
\end{equation}
where, for each $j \in N \setminus \{n\}$ 
\begin{gather}
\label{psi_exp}
\psi_{ij}=
\begin{cases}
(n-j) \; c(n,j) & \mbox{if } i \leq j 
\\
-j  \; c(n,j) & \mbox{if } i > j
\end{cases}
\\
\notag
\mbox{and} \qquad c(n,j)=\frac{2n - 3j - j^2}{2n (j+1) (j+2)} \; .
\end{gather}
\end{thm}
\begin{proof}
See the Appendix
\end{proof}

Let us denote by $\mathbf{v}=(v_j - v_{j+1})^T_{j \in N \setminus \{n\}}$ the vector of differences between the evaluations of adjacent players, and by $\mathbf{\Psi}$, the $n \times (n-1)$ matrix of the linear coefficients for the Shapley value. 
According to Theorem \ref{linshap_expl}, the Shapley value can be expressed in matricial notation as
\begin{equation}
\label{shapmatprod}
\boldsymbol{\phi}(v_g)=\mathbf{\Psi}\cdot \mathbf{v} \; ,
\end{equation}
with $\boldsymbol{\phi}(v_g)=(\phi_1(v_g),\ldots,\phi_n(v_g))^T$. Moreover, the $j$-th column in $\mathbf{\Psi}$ has null sum, and its first $j$ elements are all equal to each other, and so are its last $n-j$ elements, with a gap between the two groups of coefficient given by
\begin{equation}
\label{gappsi}
\psi_{jj} - \psi_{j+1,j}=c(n,j) \: .
\end{equation}

\begin{xmpl}[label=exa:cont]
Consider a situation with 5 agents whose evaluations of the object are 10, 6, 3, 2, 1, respectively. 
The values of the differences according to \eqref{gappsi} are: 
$$\psi_{11}-\psi_{21} = 0.1; \psi_{22}-\psi_{32} = 0.00; \psi_{33}-\psi_{43} = -0.04; \psi_{44}-\psi_{54} = -0.06$$
By  \eqref{psi_exp}, we can compute the matrix of coefficients:
$$\mathbf{\Psi} = \begin{bmatrix}
0.080&0.000&-0.016&-0.012 \\
\ -0.020\ &\ 0.000\ &\ -0.016\ &\ -0.012\  \\
-0.020&0.000&-0.016&-0.012 \\
-0.020&0.000&0.024&-0.012 \\
-0.020&0.000&0.024&0.048
\end{bmatrix} \; .$$
The vector of the differences of the evaluations is $\mathbf{v}=(4,3,1,1)^T$, and the Shapley value is $\boldsymbol{\phi}(v_g) = \mathbf{\Psi} \cdot \mathbf{v} = (0.292, -0.108, -0.108, -0.068, -0.008)^T$. \\
{Note that if the evaluations of the object were 10, 9, 8, 4, 1, respectively, the matrix $\mathbf{\Psi}$ would be the same, while $\mathbf{v} = (1,1,4,3)^T$ and the Shapley value would be $\phi(v_g) = (-0.020, -0.120, -0.120, 0.040, 0.220)^T$, i.e. the signs for players 1, 4, 5 would change, while players 2 and 3 would again be negative.}
\end{xmpl}

The proof of Theorem \ref{linshap_expl} makes use of the following combinatorial result which we believe is of autonomous beauty and importance.

\begin{lem}
\label{lemmabinom1}
For every $j,t\in \N$,
\begin{gather}
   \sum_{s=1}^t\frac{\binom{t}{s}}{\binom{j+t}{s}} = \frac{t}{j+1}\,;\label{induz1}\\
   \sum_{s=1}^t\frac{s\,\binom{t}{s}}{\binom{j+t}{s}} = \frac{t\,(j+t+1)}{(j+1)(j+2)}\,.\label{induz2}
\end{gather}
\end{lem}
\begin{proof}
Let us start by proving \eqref{induz1} by induction on $t$. It is trivially true for $t=1$. Assume now it is true for $t$, let us show that it is true for $t+1$.
Indeed
\begin{multline*}
\sum_{s=1}^{t+1}\frac{\binom{t+1}{s}}{\binom{j+t+1}{s}}
=
\sum_{s=1}^{t+1}\frac{\frac{t+1}{s}\binom{t}{s-1}}{\frac{j+t+1}{s}\binom{j+t}{s-1}}
=
\frac{t+1}{j+t+1}\sum_{s=1}^{t+1}\frac{\binom{t}{s-1}}{\binom{j+t}{s-1}} =
\\
=
\frac{t+1}{j+t+1}\sum_{s=0}^{t}\frac{\binom{t}{s}}{\binom{j+t}{s}}
=
\frac{t+1}{j+t+1}\,\left[1 + \frac{t}{j+1}\right]
=
\frac{t+1}{j+1}\,,
\end{multline*}
by the induction hypothesis. This is precisely \eqref{induz1} with $t+1$ replacing $t$.

We now prove \eqref{induz2}, again by induction on $t$. Since the formula is trivial for $t=1$, let us assume it holds for $t$. Then, reasoning as above,
\begin{multline*}
\sum_{s=1}^{t+1}\frac{s\,\binom{t+1}{s}}{\binom{j+t+1}{s}}
=
\frac{t+1}{j+t+1}\sum_{s=0}^{t}\frac{(s+1)\,\binom{t}{s}}{\binom{j+t}{s}}
=
\frac{t+1}{t+j+1}\,\left[1 + \sum_{s=1}^t\frac{s\,\binom{t}{s}}{\binom{j+t}{s}}
+ \sum_{s=1}^t\frac{\binom{t}{s}}{\binom{j+t}{s}}\right]\,,
\end{multline*}
which, using \eqref{induz1} and the induction hypothesis, gives
$$
\sum_{s=1}^{t+1}\frac{s\,\binom{t+1}{s}}{\binom{j+t+1}{s}}
=
\frac{t+1}{t+j+1}\,\left[1 + \frac{t\,(j+t+1)}{(j+1)(j+2)} + \frac{t}{j+1}
\right] 
= \frac{(t+1)\,(j+t+2)}{(j+1)(j+2)}\,.
$$
This completes the proof of the Lemma.
\end{proof}

We end this section showing that the computational burden for evaluating the Shapley value is polynomial in the number of agents.
\begin{cor}
\label{cor_comp}
The Shapley value is computed in $O(n^2)$ time
\end{cor}
\begin{proof}
The matrix $\mathbf{\Psi}$ requires the computation of the gaps \eqref{gappsi} i.e. $n-1$ steps; then it is necessary to compute two values for each column corresponding to the closed formula \eqref{psi_exp}, i.e. $2(n-1)$ steps; and finally, the scalar product $\mathbf{\Psi} \cdot \mathbf{v}$ requires $O(n^2)$ operations. Thus the overall computational complexity is $O(n^2)$.
\end{proof}

\section{Constant Attitude towards Collusion}
\label{attitude}
We will now investigate about the presence of agents with constant attitude toward collusion. 
According to Theorem \ref{linshap_expl}, the Shapley value depends primarily on the ranking position of the agent evaluation. We will now show that  players in the extreme positions in the ranking play a  prominent role in the colluding process when compared to the intermediate ones. Moreover, in  any instance of the game, one or two ``central'' players will constantly inhibit collusion and their behavior will depend only on the ranking of their evaluation, not on their absolute value.
\begin{xmpl}[continues=exa:cont] 
For any determination of the evaluations, the Shapley value for players 2 and 3 will be non-positive and identical, since they are given by the scalar product of corresponding identical non-positive rows in the matrix $\mathbf{\Psi}$ with the non-negative vector $\mathbf{v}$. It is easy to verify that each time we move up or down from rows 2 and 3, a negative coefficient of the row turns positive, implying an increase in the Shapley value of the corresponding player.
\end{xmpl}
We now investigate the presence of players whose attitude towards collusion is determined exclusively by the ranking position of the corresponding evaluation.
\begin{defn}
Player $i$ is
\begin{itemize}
\item {\em Strongly (Weakly) collusion prone} w.r.t.\ any specification of the evaluations compatible with the given ranking if $\psi_{ij} > 0$ ($\psi_{ij} \geq 0$) for every $j \in N \setminus \{n\}$;
\item {\em Strongly (Weakly) collusion averse} w.r.t.\ any specification of the evaluations compatible with the given ranking if $\psi_{ij} < 0$ ($\psi_{ij} \leq 0$) for every $j \in N \setminus \{n\}$.
\end{itemize}
\end{defn}
Example \ref{exa:cont} showed that in the case of $n=5$ agents, agents 2 and 3 are weakly collusion averse, and their common Shapley value is a global minimum for $\phi_j(v_g)$ in the variable $j$, since in the other rows of $\mathbf{\Psi}$ one or more negative coefficients turn positive. Also, there is no collusion prone player in the same situation.

A thorough examination of similar situations involving small number of agents, reveals the constant presence of collusion averse players (weakly or strongly).
\begin{xmpl}
If we denote with a `` $+$'', `` $0$'' or a `` $-$'', respectively, a positive, null or negative coefficient in $\mathbf{\Psi}$ for the cases with $n=2$ players through $n=6$ players, we find the following evidence
\begin{center}
\begin{tabular}{ccc}
$\kbordermatrix{
&   \\
1 & 0 \\
2 & 0 \\
}$
&
$\kbordermatrix{
& &  \\
 & +&- \\
2  & -&- \\
 & -&+ \\
}$
&
$\kbordermatrix{
& & &  \\
 & +&-&- \\
2 & -&-&- \\
 & -&+&- \\
 & -&+&+ \\
} $
\\
$n=2$ & $n=3$ & $n=4$\\ 
 & & \\
$\kbordermatrix{
&& &&  \\
 & +&0&-&- \\
2 & -&0&-&- \\
3& -&0&-&- \\
 & -&0&+&- \\
 & -&0&+&+ \\
}$ &   & 
$\kbordermatrix{
& & & & & \\
 & +&+&- &- &-\\
 & -&+&- & -& -\\
3 & -&-&- & - &-\\
 & -&-&+ & - &-\\
 & -&-&+ & + &-\\
  & -&-&+ & + &+\\
  } $
\\
 $n=5$ & & $n=6$
\end{tabular}  .
\end{center}
When $n=2$, both players are simultaneously weakly averse and weakly prone (actually the Shapley value for both is always null). Starting from $n=3$ there never exists a collusion prone agent, while there exists at least a collusion averse one. In the case $n=5$ there are two weakly averse players (as already seen in Example \ref{exa:cont}), while there is a single strongly collusion averse one in all the other cases. In any case, the Shapley value reaches a minimum at the collusion averse player(s).

The alternation between one strongly averse player and two weakly averse ones is confirmed by the computation of the coefficients' signs up to $n=15$ players in Table \ref{table:signpsi}.
\begin{table}[h]
\caption{The pattern of doomed players: strongly (weakly, resp.) doomed players are indicated by $\otimes$ ($\odot$, resp.)}
\label{table:signpsi}
\centering
\begin{tabular}{|c||ccccccccccccccc|}\hline 
Pl & 1 & 2 & 3 & 4 & 5 & 6 & 7 & 8 & 9 & 10 & 11 & 
12 & 13 & 14     \\ \hline \hline
 2 &  $\odot$ & $\odot$ &  &  &  &  &  &  &  &  &  &  &  &      \\ \hline 
  3 &  $\cdot$ & $\otimes$ & $\cdot$ &  &  &  &  &  &  &  &  &  &  &      \\ \hline 
  4 &  $\cdot$ & $\otimes$ & $\cdot$ & $\cdot$ &  &  &  &  &  &  &  &  &  &      \\ \hline 
5 &  $\cdot$ & $\odot$ & $\odot$ & $\cdot$ &  $\cdot$ &  &  &  &  &  &  &  &  &     \\ \hline
6 &  $\cdot$ & $\cdot$ & $\otimes$ & $\cdot$ & $\cdot$&  $\cdot$     &  &  &  &  &  &  &  &    \\ \hline
7 &  $\cdot$ & $\cdot$ & $\otimes$ & $\cdot$ & $\cdot$&  $\cdot$   & $\cdot$ &  &  &  &  &  &  &     \\ \hline
8 &  $\cdot$ & $\cdot$ & $\otimes$ & $\cdot$ & $\cdot$&  $\cdot$   & $\cdot$ & $\cdot$  &  &  &  &  &  &     \\ \hline
9 &  $\cdot$ & $\cdot$ & $\odot$ & $\odot$ & $\cdot$&  $\cdot$   & $\cdot$ & $\cdot$  & $\cdot$ &  &  &  &  &     \\ \hline
10 &  $\cdot$ & $\cdot$ & $\cdot$ &$\otimes$ & $\cdot$ & $\cdot$&  $\cdot$   & $\cdot$ & $\cdot$  & $\cdot$ &  &  &  &     \\ \hline
11 &  $\cdot$ & $\cdot$ & $\cdot$ &$\otimes$ & $\cdot$ & $\cdot$&  $\cdot$   & $\cdot$ & $\cdot$  & $\cdot$  & $\cdot$ &  &   &  \\ \hline
12 &  $\cdot$ & $\cdot$ & $\cdot$ &$\otimes$ & $\cdot$ & $\cdot$&  $\cdot$   & $\cdot$ & $\cdot$  & $\cdot$  & $\cdot$ & $\cdot$ &   &  \\ \hline
13 &  $\cdot$ & $\cdot$ & $\cdot$ &$\otimes$ & $\cdot$ & $\cdot$&  $\cdot$   & $\cdot$ & $\cdot$  & $\cdot$  & $\cdot$ & $\cdot$ &  $\cdot$ &  \\ \hline
14 &  $\cdot$ & $\cdot$ & $\cdot$ &$\odot$ & $\odot$ & $\cdot$&  $\cdot$   & $\cdot$ & $\cdot$  & $\cdot$  & $\cdot$ & $\cdot$ &  $\cdot$ & $\cdot$ \\ \hline
15 &  $\cdot$ & $\cdot$ & $\cdot$ & $\cdot$ &$\otimes$ & $\cdot$ & $\cdot$&  $\cdot$   & $\cdot$ & $\cdot$  & $\cdot$  & $\cdot$ & $\cdot$ &  $\cdot$  \\ \hline
\end{tabular}
\end{table}
Note that the gap between the cases with two weakly averse players always increases by one, and these occurrences mark a change in position of the strongly averse player in all the other occurences (i.e.\ the gaps between these occurrences), so that the averse players form a ``ladder'' pattern in the table.
\end{xmpl}

According to the pattern in Table \ref{table:signpsi}, two weakly averse players are present when $n_2=2$, when $n_3=2 +3 = 5$ and, in general, when, for some $k \geq 2$,
\begin{equation}
\label{formula:nk}
n_k=\sum_{j=2}^{k} j = \frac{k^2 + k - 2 }{2} \; .
\end{equation}
The general pattern is confirmed by the following theorem.
\begin{thm}
\label{thm:pattern}
\begin{enumerate}[a)]
\item If, for some integer $k \geq 1$,
\begin{equation}
\label{formula:neqnk}
n=n_k \; ,
\end{equation}
with $n_k$ defined in \eqref{formula:nk}, then players k and k+1 are weakly collusion averse, and their Shapley value coincide;

If, instead,
\begin{equation}
\label{formula:nkbetween}
n_k < n < n_{k+1} \; ,
\end{equation}
then, player $k+1$ is the only strongly collusion averse player;
\item If $n=2$ both players are simultaneously weakly collusion averse and prone, while for $n>2$ there is no collusion prone players;
\item The Shapley value is decreasing in the agent, from player 1 up to the first collusion averse player, and is increasing in the agent from the last  collusion averse player up to player $n$.
\end{enumerate}
\end{thm}
\begin{proof}
When $n=2$. Here $\mathbf{\Psi}=[0,0]^T$ and both players are simultaneously weakly collusion prone and averse, satisfying both $a)$ (with $k=1$) and $b)$.

When $n>2$, the sign of the gap $\psi_{jj} - \psi_{j+1,j}$ depends exclusively on the term $2n-3j - j^2$. Now the quadratic function $-x^2 - 3x + 2n$ in the real variable $x$ has a positive root
\[
x^*=\frac{-3+\sqrt{8n+9}}{2} \; ,
\]
so that the function is positive before the root and negative after it. The root is an integer if and only if $n = (k^2 + 3k)/2$
for some integer $k > 1$. We therefore distinguish between two cases
\begin{enumerate}[i)]
\item $n=n_k$ for some $k>1$. Here $x^*=k$ and the gap 
\[
\psi_{jj} - \psi_{j+1,j} \quad \mbox{is} \quad \begin{cases}
>0 & \mbox{if }j<k
\\
=0 & \mbox{if }j=k
\\
<0 & \mbox{if }j > k
\end{cases},
\]
leading to the following sign structure for $\mathbf{\Psi}$
\begin{equation}
\label{mat:sign_i}
\kbordermatrix{
 & 1 & 2 & \cdots & k - 1 & k & k+1 & \cdots & n-2 & n-1
 \\
 1 & + & + & \cdots & +  & 0 & - &\cdots & - & -
  \\
   2 & - & + & \cdots & +  & 0 & - &\cdots & - & -
  \\
 \vdots  &\vdots & \vdots &  \ddots& \vdots & \vdots& \vdots & \ddots & \vdots & \vdots
\\
k-1& - & - &  \cdots & +  &0 &  - & \cdots & - & -
\\
\mathbf{k}& \mathbf{-} & \mathbf{-} &  \cdots & \mathbf{-}  &\mathbf{0} &  \mathbf{-} & \cdots & \mathbf{-} & \mathbf{-}
\\
\mathbf{k+1}& \mathbf{-} & \mathbf{-} &  \cdots & \mathbf{-}  &\mathbf{0} &  \mathbf{-} & \cdots & \mathbf{-} & \mathbf{-}
\\
k+2 & - & - &  \cdots & -  &0 &  + & \cdots & - & -
\\
 \vdots  &\vdots & \vdots &  \ddots& \vdots & \vdots& \vdots & \ddots & \vdots & \vdots
\\
 n-1 & - & - & \cdots & -  & 0 & + &\cdots & + & -
  \\
 n & - & - & \cdots & -  & 0 & + &\cdots & + & +
} \; .
\end{equation}
Therefore players $k$ and $k+1$ (with bold typeface in the matrix) will be weakly collusion averse, while no player will be collusion prone, since every row has at least a ``$-$'' element. Moreover, rows $k$ and $k+1$ are identical, leading to the same Shapley value for the corresponding players.
\item If $n_k < n < n_{k+1}$ for some $k \geq 1$, then $k < x^* < k+1$, and the gap
\[
\psi_{jj} - \psi_{j+1,j} \quad \mbox{is} \quad \begin{cases}
>0 & \mbox{if }j \leq k
\\
<0 & \mbox{if }j \geq k+1
\end{cases},
\]
leading to the following sign structure for $\mathbf{\Psi}$ 
\begin{equation}
\label{mat:sign_ii}
\kbordermatrix{
 & 1 & 2 & \cdots  & k & k+1 & \cdots & n-2 & n-1
 \\
 1 & + & + & \cdots & +  &  - &\cdots & - & -
  \\
   2 & - & + & \cdots & +  &  - &\cdots & - & -
  \\
 \vdots  &\vdots & \vdots &  \ddots & \vdots& \vdots & \ddots & \vdots & \vdots
\\
k& - & - &  \cdots & +   &  - & \cdots & - & -
\\
\mathbf{k+1}& \mathbf{-} & \mathbf{-} &  \cdots & \mathbf{-}  &  \mathbf{-} & \cdots & \mathbf{-} & \mathbf{-}
\\
k+2 & - & - &  \cdots & -   &  + & \cdots & - & -
\\
 \vdots  &\vdots & \vdots &  \ddots & \vdots& \vdots & \ddots & \vdots & \vdots
\\
 n-1 & - & - & \cdots & -  & + &\cdots & + & -
  \\
 n & - & - & \cdots & -   & + &\cdots & + & +
} \; ,
\end{equation}
and player $k+1$ (with bold typeface in the matrix) is strongly collusion averse, while no player will be collusion prone.
\end{enumerate}
To prove $c)$ note that in both \eqref{mat:sign_i} and \eqref{mat:sign_ii}, each time we move up or down from the rows in $\mathbf{\Psi}$ corresponding to collusion averse players, a ``$-$'' element is turned into ``$+$'', increasing the Shapley value of the corresponding player.
\end{proof}

\begin{rem}
When $n=2$ the Shapley value of both players is null, while if $n=n_k$ for some integer $k>1$, the Shapley value of all players is null if the evaluations of the first $k$ players coincide and those of the last $n-k$ players coincide too.
\end{rem}

\section{Coalitions of Maximal Gain}
\label{maximal}

Most often, in TU games players eventually join the grand coalition, and the Shapley value provides a natural way for them to share its worth. The gain game that we are considering does not share this feature, since the null worth of the grand coalition is usually outdone by smaller coalitions that obtain a positive gain from collusion at the expense of agents outside the coalition.

We already formulated an alternative interpretation for the Shapley value in this context, but one question remains open: what coalition will actually prevail? We provide a first answer in terms of total or per-capita maximal gain. We will show that these coalitions will involve only a fraction $2^{-1} + o\left(n^{-1}\right)$ or $o\left(n^{-1/2}\right)$ of the population, respectively, for the total or the per-capita maximal gain.

\begin{defn}
For any $s \in N$, we define $S_s$ to be a coalition of cardinality $s$ with maximal worth among the coalitions with the same cardinality:
\[
v_g(S_s) = \max_{S, |S|=s} v_g(S) \; .
\]
Among these coalitions we will pick one with the (total) maximal gain
\begin{equation*}
\label{defmaxgain}
v_g(S_{s^*}) = \max_{s \in N} v_g(S_s) = \max_{S \subseteq N} v_g(S) \; .
\end{equation*}

\end{defn}

If we fix the number of players, a coalition yields the highest gain whenever it is formed by Player 1 together with the last players.
The following result is immediate, and we omit its proof.
\begin{prop}
\[S_s= \begin{cases}
\{1\} & \mbox{if } s=1
\\
\{ 1,n+2-s,n+3-s,\ldots,n \} & \mbox{otherwise}
\end{cases} \: .
\]
Moreover,
\[
v(S_s)=\cfrac{n-s}{n^2} \sum_{i \in \hat{S}_s }(v_1 - v_i) \; ,
\]
where $\hat{S}_s = S_s \setminus \{1\}=\{n+2-s,n+1-s,n-s,\dots,n\}$.

\end{prop}

We now consider the difference in gain between maximal coalitions whose cardinality differ by one unit:
\begin{equation}
\label{defDelta}
\Delta(s)=v(S_s) - v(S_{s-1}) \qquad s=2,3,\dots,n,
\end{equation}
and give an explicit formula for it.
\begin{lem} For any $s=2,3,\dots,n$,
\begin{equation}
\label{Deltas}
\Delta(s)=\frac{1}{n^2} \left[ 
(n-2s+2)(v_1 - v_{n+2-s})- \sum_{i \in \hat{S}_{s-1}} (v_{n+2-s} - v_i) \right] \; .
\end{equation}
\end{lem}
\begin{proof}
If $s \geq 3$, then
\begin{multline*}
\Delta(s)=\dfrac{n-s}{n^2} \sum_{i \in \hat{S}_{s}}(v_1 - v_i) - \dfrac{n+1-s}{n^2} \sum_{i \in \hat{S}_{s-1}}(v_1 - v_i) =\\
\dfrac{1}{n^2} \left[ 
(n-s)(v_1 - v_{n+2-s}) - \sum_{i \in \hat{S}_{s-1}}(v_1 - v_i) \right] =\\
\dfrac{1}{n^2} \left[ 
(n-s)(v_1 - v_{n+2-s}) - (s-2) v_1 + \sum_{i \in \hat{S}_{s-1}} v_i + (s-2)v_{n+2-s} -  (s-2)v_{n+2-s}\right] = \\
\dfrac{1}{n^2} \left[ 
(n-2s+2)(v_1 - v_{n+2-s}) - \sum_{i \in \hat{S}_{s-1}} (v_{n+2-s}-v_i) \right] \; .
\end{multline*}
For $s=2$, since $v(S_1)=0$,
\[
\Delta(2)=v(S_2)-v(S_1)=v(S_2) = \dfrac{n-2}{n^2}(v_1 - v_n) \; ,
\]
which delivers \eqref{Deltas}, once we note that $\hat{S}_1=\varnothing$ and the second addend is null.
\end{proof}
The following are straightforward consequences of the Lemma.
\begin{prop}
\begin{enumerate}[i)]
\item The following holds:
\begin{equation}
\label{gainDelta}
v(S_s)= \begin{cases}
0 & \mbox{if }s=1
\\
\sum_{j=2}^s \Delta(j) & \mbox{otherwise}
\end{cases} \; ;
\end{equation}
\item $\Delta(s)$ is nonincreasing in $s$ ;
\item If $s^*$ is the largest integer such that $\Delta(s^*) \geq 0$ then $S_{s^*}$ will be a coalition with total maximal gain.
\end{enumerate}
\end{prop}
\begin{proof}
Statements $i)$ and $ii)$ are straightforward consequences of equations \eqref{defDelta} and \eqref{Deltas}, respectively. To prove \eqref{defmaxgain} note that $v(S_s)$ is maximal when it contains the largest number of positive increments. If we allow for null increments as well, we obtain that $s^*$ is the largest integer for which $v(S_s)$ is maximal. 
\end{proof}
Note that the following holds:
\begin{gather*}
\Delta(2)>0 \qquad \mbox{if }v_1 \neq v_n \; \;
\\
\Delta(n)<0 \qquad \mbox{if } v_2 \neq v_n \; ,
\end{gather*}
and the total maximal gain criterion suggests the following coalition formation mechanism, when agents know their evaluations' ranking (but not their exact values): the first and the last agents coalesce, then, in turn, agents $n-1,n-2,\ldots$ are invited to join the coalition as long as the corresponding increment $\Delta$ is positive, or it does not become too small.

The per-capita gain of a coalition $S \subseteq N$ is defined by
\[
v_{pc}(S)=\frac{v(S)}{|S|} .
\]
As before, the per-capita gain of a coalition with $s$ participants is maximal with $S_s$, and we consider the difference between maximal coalitions that differ by a unit:
\[
\delta(s)=v_{pc}(S_s) - v_{pc}(S_{s-1}) \qquad s=2,3,\dots,n,
\]
with
\begin{equation}
\label{gainDelta}
v_{pc}(S_s)= \begin{cases}
0 & \mbox{if }s=1
\\
\sum_{j=2}^s \delta(j) & \mbox{otherwise}
\end{cases} \; .
\end{equation}
We provide a characterization of the increment $\delta(s)$ in terms of the increments of the (absolute) gain.
{\sc
\begin{lem}
For any $s=2,3,\dots,n$ the increment $\delta(s)$ may be written as
\begin{equation}
\label{deltachar}
\delta(s)=\begin{cases}
\dfrac{\Delta(2)}{2} & \mbox{if }s=2
\\
\dfrac{\Delta(s)-\sum_{j=2}^{s-1} \left( \Delta(j)-\Delta(s) \right)}{s(s-1)} & \mbox{otherwise}
\end{cases} \; .
\end{equation}
Alternatively, it can be written as:
\begin{equation}
\label{smalldeltas}
\delta(s)=\frac{(n+s-s^2) \left( v_1 - v_{n+2-s}   \right) - n \sum_{j \in \hat{S}_{s-1}} \left(  v_{n+2 - s} -v_j \right)}{n^2 s (s-1)} \; .
\end{equation}
\end{lem}
}
\begin{proof} First of all we prove \eqref{deltachar}.
If $s=2$ then
\[
\delta(2)=\cfrac{v(S_2)}{2} - v(S_1) = \cfrac{\Delta(2)}{2} .
\]
If $s \geq 3$, then
\begin{multline*}
\delta(s)=\cfrac{v(S_s)}{s} - \cfrac{v(S_{s-1})}{s-1} = 
\cfrac{v(S_{s-1})+\Delta(s)}{s} - \cfrac{v(S_{s-1})}{s-1} =
\\
\cfrac{(s-1)\Delta(s) - v(S_{s-1})}{s(s-1)} = 
\cfrac{(s-1)\Delta(s) - \sum_{j=2}^{s-1} \Delta(j) }{s(s-1)} =
\\
\cfrac{\Delta(s)- \sum_{j=2}^{s-1} \left( \Delta(j)-\Delta(s) \right) }{s(s-1)} \; .
\end{multline*}
To prove \eqref{smalldeltas}, consider the following chain of equations:
\begin{multline*}
\delta(s)=\frac{1}{n^2} \left[ \frac{n-s}{s^2} \sum_{i \in \hat{S}_s} (v_1 - v_i)  - \frac{n-s+1}{s-1} \sum_{j \in \hat{S}_{s-1}}  (v_1 - v_j) \right] =
\\
\frac{1}{n^2 s (s-1)}  \left[ (n-s)(s-1) \sum_{i \in \hat{S}_s} (v_1 - v_i) - s(n-s+1) \sum_{j \in \hat{S}_{s-1}} (v_1 - v_j) \right] =
\\
\frac{1}{n^2 s (s-1)}  \left[  (n-s)(s-1)(v_1 - v_{n+2-s}) - n \sum_{j \in \hat{S}_{s-1}} (v_1 - v_j) \right] =
\\
\frac{(n-s)(s-1)(v_1 - v_{n+2-s}) - n(s-2) v_1 + n \sum_{j \in \hat{S}_{s-1}} v_j -n(s-2) v_{n+2-s} + n(s-2) v_{n+2-s} }{n^2 s (s-1)} =
\\
\frac{(n+s-s^2) \left( v_1 - v_{n+2-s}   \right) - n \sum_{j \in \hat{S}_{s-1}} \left(  v_{n+2 - s} -v_j \right)}{n^2 s (s-1)} \; .
\end{multline*}
\end{proof}

\begin{prop}
Suppose $v_1 \neq v_n$, then there exists some index $s^{**}$ such that
\begin{gather}
\label{signdelta}
\delta(s) = \begin{cases}
\geq 0 & \mbox{if } j \leq s^{**}
\\
< 0  & \mbox{if } j > s^{**}
\end{cases}
\\
\label{maxgainpc}
v_{pc}(S_{s^{**}}) = \max_{S \subseteq N} v_{pc}(S) \; ,
\end{gather}
and $s^{**}$ is the largest integer for which $v_{pc}(S_s)$ is maximal.
\end{prop}
\begin{proof}
Clearly $\delta(2)>0$ when $v_1 \neq v_n$.
The numerator in both expressions of \eqref{deltachar} is decreasing in $s$, and it cannot be always positive, otherwise we obtain the contradiction $v_{pc}(N)>0$. \eqref{deltachar} is decreasing in $s$, and it cannot be always positive, otherwise we obtain the contradiction $v_{pc}(N)>0$. 

To prove \eqref{maxgainpc} note that $v_{pc}(S_s)$ is maximal when it contains the largest number of positive increments. If we allow for null increments as well, we obtain that $s^{**}$ is the largest integer for which $v_{pc}(S_s)$ is maximal. 
\end{proof}

If the per-capita maximal  worth is pursued, agents will adopt the coalition formation procedure already described with the increment $\delta$ in place of $\Delta$ as a measure of the added value of each newcomer to the coalition already formed. Since, however, the increment $\delta$ may be non-monotonic, the procedure will stop only when $\delta$ becomes negative.

We now establish some bounds for $s^*$ and $s^{**}$.
\begin{prop}If $v_1 \neq v_n$ then
\begin{gather}
\label{bounds1}
s^{**} \leq s^*
\\
\label{bounds2}
2 \leq s^* \leq 
\left\lfloor \dfrac{n}{2}+1 \right\rfloor 
\\
\label{bounds3}
2 \leq s^{**} \leq 
\left\lceil \sqrt{n} \right\rceil  \; .
\end{gather}
\end{prop}
\begin{proof}[Outline of the proof]
Since $v_{pc}(S_1)=0$ then  $s^{**} \geq 2$. When $\Delta(s) < 0$, then $\delta(s) < 0$, which implies \eqref{bounds1}.

To establish \eqref{bounds2}, consider formula \eqref{Deltas}. When $s > \frac{n}{2}+1$, then, either $v_1 \neq v_{n+2-s}$ and the first addend in \eqref{Deltas} is strictly negative, or  $v_{n+2-s} \neq v_n$ and the second addend in \eqref{Deltas} is strictly negative. Either way $s^* < s$, therefore $s \leq \left\lfloor \frac{n}{2} + 1 \right\rfloor$.

In  a similar fashion if $s > \frac{1 + \sqrt{1+4n}}{2}$, then $\delta(s)< 0$. Now  $\frac{1 + \sqrt{1+4n}}{2} > \sqrt{n}$ and $\frac{1 + \sqrt{1+4n}}{2} - \sqrt{n} < 1$ when $n \geq 1$. Therefore $s^{**} \leq \left\lceil \sqrt{n} \right\rceil = \left\lfloor \frac{1 + \sqrt{1+4n}}{2} \right\rfloor$.
\end{proof}
The following examples prove that the bounds are tight for $s^*$.
\begin{xmpl}
Suppose $v_1=v_2= \cdots=v_{n-1}=1$ and $v_{n}=0$. Then
\[
v(S_s)= \begin{cases}
0 & \mbox{if }s=1
\\
\cfrac{n-s}{n^2} & \mbox{if } s \geq 2 \; ,
\end{cases}
\] which is maximal for $s=2$. Therefore $s^{*}=s^{**}=2$.
\end{xmpl}

\begin{xmpl}
Suppose $n$ even, $v_1=0$ and $v_2 = v_3 = \cdots = v_n$. Then, 
\[
v(S_s)= 
\cfrac{(n-s)(s-1)}{n^2} \; ,
\] which is maximal for $s=\frac{n}{2}$ and $s^*=\frac{n}{2}+1$.
Also, either $s^{**}=  \left\lceil \sqrt{n} \right\rceil $ or $s^{**}=  \left\lfloor\sqrt{n} \right\rfloor $. 
\end{xmpl}

\begin{rem}
For some instances of the game the bounds in \eqref{bounds2} can be made sharper. For instance:
\begin{itemize}
\item If $n$ is even, and $v_1 \neq v_{n/2}$, then $s^{*} \leq \frac{n}{2}$;
\item If for some $\ell \leq \frac{n}{2}+1$, $v_1 = v_{n + 2 - \ell}$ and $v_{n+2 - \ell} > v_n$, then $s^* < \ell$.
\end{itemize}
\end{rem}
\section{Concluding Remarks}
\label{concl}

In this paper we dealt with the explicit computation of the Shapley value for the gain game introduced in Briata, Dall'Aglio and Fragnelli (2012), and its consequences to the players'attitude towards collusion in the Knaster fair division procedure. 

The Shapley value can be obtained as the matrix product among a matrix whose entries depend only on the number of agents and the vector of the differences in the valuations of the agents.

Further research may better formalize the coalition formation process, coupling the players' desire to join coalition with highest worth, together with some notion of stability.
Furthermore, we intend to compare the findings on the Shapley value of this game with other classical game theoretical solutions, on the basis of their characterizing properties.

\section*{Dedication and acknowledgment} This work is dedicated to the memory of Mario Dall'Aglio, researcher and educator.

 The authors would like to thank Stefano Moretti  for suggesting equation \eqref{shapdiff} in the Appendix, and two anonimous referees for their constructive advices.

\appendix
\section{Appendix: Proof of Theorem \ref{linshap_expl}}

Finding an explicit expression for the Shapley value is hampered by the fact that the marginal contribution of an agent to a coalition has a complicated espression. In fact, if $i \in N$ and  $S \subseteq N \setminus \{i\}$, the following holds
\begin{equation}
\label{incr_form}
v_g(S \cup \{i\})-v_g(S) =
 \cfrac{n-s-1}{n^2} \; \max \left\{ s(v_i - b^S),b^S - v_i  \right\} - \cfrac{1}{n^2} \sum_{k \in S}(b^S-v_k) \; .
\end{equation}

We will consider, instead, a general formula about the difference of Shapley values for adjacent players
  \begin{equation}
  \label{shapdiff}
  \phi_j(v_g) - \phi_{j+1}(v_g) = 
  \sum_{S \subseteq N \setminus \{ j,j+1 \}} \frac{s!(n-s-2)!}{(n-1)!} \left[ v_g(S \cup \{j\}) - v_g(S \cup \{j+1\}) \right] \: .
  \end{equation}
   The following lemma shows that the difference between the value of the coalition joined by two successive players results in a formula much simpler than \eqref{incr_form}. 
   For any $j \in N \setminus \{n\}$, let $J= \{1,\ldots,j\}$ and $J^c= \{j+1,\ldots,n\}$.

\begin{lem}
\label{lem_vdiff_adj}
For any $j \in N \setminus \{n\}$ and $S \subseteq N \setminus \{j,j+1\}$,
\begin{equation}
\label{eq_vdiff_adj}
v_g(S \cup \{j\}) - v_g(S \cup \{j+1\})=
\begin{cases}
\frac{(n-s-1)s}{n^2}(v_j-v_{j+1}) & \mbox{if } S \cap J = \varnothing
\\
-\frac{n-s-1}{n^2}(v_j-v_{j+1})  & \mbox{if } S \cap J \neq \varnothing
\end{cases}\; .
\end{equation}

\end{lem}

\begin{proof} 
 If $S \cap J = \varnothing$, then $S \subseteq J^c \setminus \{j+1\}$, and
 \begin{multline*}
 v_g(S \cup \{j\}) - v_g(S \cup \{j+1\})=\frac{n-s-1}{n^2} \left(
 \sum_{i \in S \cup \{j\}} \left( b_{S \cup \{j\}} - v_i \right)
 - \sum_{i \in S \cup \{j+1\}} \left( b_{S \cup \{j+1\}} - v_i \right)\right) =
 \\
 \frac{n-s-1}{n^2} \left(
 \sum_{i \in S } \left( v_j - v_i \right)
 - \sum_{i \in S } \left( v_{j+1} - v_i \right)\right) =
 \frac{n-s-1}{n^2} s (v_j - v_{j+1}) \; .
 \end{multline*}
Otherwise, $S \cap J \neq \varnothing$, and
 \begin{multline*}
 v_g(S \cup \{j\}) - v_g(S \cup \{j+1\})=\frac{n-s-1}{n^2} \left(
 \sum_{i \in S \cup \{j\}} \left( b_S - v_i \right)
 - \sum_{i \in S \cup \{j+1\}} \left( b_S - v_i \right)\right) =
 \\
 - \frac{n-s-1}{n^2} (v_j - v_{j+1}) \; .
 \end{multline*}
 \end{proof}

\begin{proof}[Proof of Theorem \ref{linshap_expl}]
Each coalition worth in the gain game can be written as
\begin{equation}
\label{worthlinear}
v_g(S) = \frac{n-s}{n^2}\sum_{i \in S} ( b_S - v_i ) = \sum_{j=1}^{n-1} c_{S,j} \left( v_j - v_{j+1} \right) \; ,
\end{equation}
where
\[
c_{S,j}= \frac{n-s}{n^2} \sum_{i \in S} I_{S,i}(j)\; , \qquad I_{S,i}(j)=\begin{cases}
1 & \mbox{if } \min_{h \in S} h \leq j < i
\\
0 & \mbox{otherwise}
\end{cases} \; ,
\]
and the coefficients $c_{S,j}$ do not depend on the actual values of the $v_i$'s (within a fixed ranking). The Shapley value for a player is a linear combination of differences of coalitions' worths, and therefore a linear combination of the differences in valuations between successive agents, explaining \eqref{linshap_eq}.

The gain game is such that $v_g(N)=0$ and, therefore,
$$
0 = v_g(N) = \sum_{i =1}^n \phi_i(v_g) = \sum_{i = 1}^n \sum_{j =1}^{n-1}
\psi_{ij}(v_{j} - v_{j+1}) =  \sum_{j=1}^{n-1} ( v_{j} - v_{j+1})
\sum_{i =1}^n \psi_{ij} \; .
$$
Since this holds for any choice of the $v_i$, $i \in N$, it must be that 
\begin{equation}
\label{sumpsi_zero}
\sum_{i=1}^n \psi_{ij}=0 \qquad \mbox{for any }j \in N \setminus \{n\}.
\end{equation}
In order to have a simple expression for the coefficients $\psi_{ij}$ we introduce a particular set of evaluation for the agents, and then we consider the related gain game. For any $j \in N \setminus \{n\}$, let the evaluations be as follows:
\begin{equation}
\label{simplegame}
v_h= \begin{cases}
1 & \mbox{if }h \in J
\\
0 & \mbox{if }h \in J^c
\end{cases} \; .
\end{equation}
Let $v_{g,j}$ be the corresponding gain game. Clearly, from \eqref{linshap_eq} we have:
\[
\phi_i(v_{g,j}) = \psi_{ij} \qquad \mbox{for any }i,j \in N.
 \]
 Moreover, by the symmetry\footnote{Two players $i,j \in N$ are called symmetric if $v(S\cup \{i\}) = v(S\cup \{j\})$, $S \subseteq N\setminus\{i,j\}$.} of the players in $J$ and in $J^c$, and recalling that the Shapley value assigns equal amounts to symmetric players we have
 \begin{gather}
 \label{psiequal}
 \psi_{1j} = \psi_{2j} = \cdots = \psi_{jj} =a_j
  \\
  \notag
 \psi_{j+1,j} = \psi_{j+2,j} = \cdots = \psi_{nj} =b_j \; .
   \end{gather} 
To compute every $\psi_{ij}$ we only need to determine the differences between  $\psi_{jj}$ and $\psi_{j+1,j}$, which we can write as: 
\begin{equation}
\label{eqdiffpsi}
\psi_{jj}- \psi_{j+1,j} =a_j - b_j= \phi_j(v_{g,j}) - \phi_{j+1}(v_{g,j}) \; .
\end{equation}
We now apply \eqref{shapdiff}, together with Lemma \eqref{lem_vdiff_adj}. Noting  that $v_j - v_{j+1}=1$ in the game $v_{gj}$, the Lemma distinguishes between two cases:
\begin{description}
\item[Case 1]: $S \cap J = \varnothing$ and, therefore, $S \subseteq J^c \setminus \{j+1\}$.
The part of formula \eqref{shapdiff} pertaining these coalitions is present only when $j \in N \setminus \{n-1,n\}$ and it is given by
\begin{multline*}
\sum_{S \subseteq J^c \setminus \{j+1\}} \frac{s!(n-s-2)!}{(n-1)!} \left( v_g(S \cup \{j\}) - v_g(S \cup \{j+1\})\right) =   
\\
\sum_{S \subseteq J^c \setminus \{j+1\}} \frac{s!(n-s-2)!}{(n-1)!} \frac{(n-s-1)s}{n^2} = 
\frac{1}{n^2} \sum_{S \subseteq J^c \setminus \{j+1\}}
\frac{s!(n-s-2)!(n-s-1)}{(n-1)!} s =
\\
\frac{1}{n^2} \sum_{S \subseteq J^c \setminus \{j+1\}}
\frac{1}{\binom{n-1}{s}} s =
\frac{1}{n^2} \sum_{s=1}^{n-j-1} \frac{\binom{n-j-1}{s}}{\binom{n-1}{s}} s \; .
\end{multline*}
\item[Case 2] $S \cap J  \neq \varnothing$. 
The part of formula \eqref{shapdiff} pertaining this case is given by
\begin{multline*}
\sum_{S \cap J \neq \varnothing, S \subseteq N \setminus \{j,j+1\}} -\ \frac{s! (n-s-2)!}{(n-1)!} \frac{n-s-1}{n^2} =
\\
= \ - \frac{1}{n^2} \sum_{S \cap J \neq \varnothing, S \subseteq N \setminus \{j,j+1\}} \frac{s! (n-s-1)!}{(n-1)!}  =
- \frac{1}{n^2} \sum_{S \cap J \neq \varnothing, S \subseteq N \setminus \{j,j+1\}} \frac{1}{\binom{n-1}{s}}  \; .
\end{multline*}
Now, we can choose  a set of $s$ units, with at least 1 unit from the first $j-1$ in a number of ways given by
\[
\sum_{t = \max \{1, s+j+1-n\}}^{\min \{j-1,s\}} \binom{j-1}{t} \binom{n-j-1}{s-t} = \begin{cases}
\binom{n-2}{s} - \binom{n-j-1}{s} & \mbox{if }s \leq n - j -1
\\ \ \\
\binom{n-2}{s} & \mbox{if } s > n - j -1
\end{cases} \; .
\]
Therefore, the part of \eqref{shapdiff} pertaining this case becomes
\begin{multline*}
- \frac{1}{n^2} \sum_{s=1}^{n-2} [\mbox{\# of ways}] \frac{1}{\binom{n-1}{s}} =
\begin{cases}
- \frac{1}{n^2} \sum\limits_{s=1}^{n-2}  \frac{\binom{n-2}{s}}{\binom{n-1}{s}} & \mbox{if }j = n-1
\\ \ \\
- \frac{1}{n^2} \sum\limits_{s=1}^{n-2}  \frac{\binom{n-2}{s}}{\binom{n-1}{s}} + \frac{1}{n^2} \sum\limits_{s=1}^{n-j-1}  \frac{\binom{n-j-1}{s}}{\binom{n-1}{s}}& \mbox{if }j < n-1
\end{cases} \; .
\end{multline*}
\end{description}

Joining the results for Case 1 and Case 2 when $j < n-1$, we have
\begin{equation}
\label{provv1}
\psi_{jj} - \psi_{j+1,j} = \frac{1}{n^2}\,\left[ \sum_{s=1}^{n-j-1} \frac{\binom{n-j-1}{s}}{\binom{n-1}{s}} s -  \sum_{s=1}^{n-2} \frac{\binom{n-2}{s}}{\binom{n-1}{s}}  +  \sum_{s=1}^{n-j-1} \frac{\binom{n-j-1}{s}}{\binom{n-1}{s}}\right] \,.
\end{equation}
Applying Lemma \ref{lemmabinom1} to the r.h.s. of \eqref{provv1}, it is easy to check that 
\begin{equation}
\label{signdiff}
\psi_{jj} - \psi_{j+1,j} = 
\frac{1}{n^2}\,\left[ \frac{n\,(n-j-1)}{(j+1)\,(j+2)} -  \frac{n-2}{2}  +   \frac{n-j-1}{j+1}\right]
= \frac{2n - 3j - j^2}{2n(j+1)(j+2)}.
\end{equation}
To prove \eqref{psi_exp}, we recall \eqref{sumpsi_zero}, \eqref{psiequal} and solve the following system of linear equations in the  variables $a_j$ and $b_j$
\[
\left\{
\begin{array}{l}
j a_j + (n-j) b_j =0
\\
a_j - b_j = \frac{2n - 3j - j^2}{2n(j+1)(j+2)}
\end{array}
\right. \; .
\] 
\end{proof}

\section*{References}
\begin{verse}

Brams, S.J. and Taylor, A.D. (1996). \textit{Fair-Division: From Cake Cutting to Dispute Resolution}, New York, Cambridge University Press. 

Brams, S.J. and Taylor, A.D. (1999). \textit{The WinWin Solution: Guaranteeing Fair Shares to Everybody}, New York, W. W. Norton. 

Branzei, R., Fragnelli, V., Meca, A. and Tijs, S. (2009). \textit{On Cooperative Games Related to Market Situations and Auctions}, International Game Theory Review, 11, 459-470. 

Briata, F., Dall'Aglio, M. and Fragnelli, V. (2012). \textit{Dynamic Collusion and Collusion Games in Knaster's Procedure}, AUCO Czech Economic Review {6}, 199-208.

Fragnelli, V. and Marina, M.E. (2009). \textit{Strategic Manipulations and Collusions in Knaster Procedure}, AUCO Czech Economic Review, {3}, 143-153. 

Fragnelli, V. and Meca, A. (2010). \textit{A Note on the Computation of the Shapley Value for von Neumann-Morgenstern Market Games}, International Game Theory Review, 12, 287-291. 

Gambarelli, G., Iaquinta, G. and Piazza, M. (2012). \textit{Anti-collusion indices and averages for the evaluation of performances and judges}, Journal of Sports Sciences, {30}, 411-417.

Graham, D. and Marshall, R. (1987). \textit{Collusive Bidder Behavior at a Single Object Second-Price and English Auctions}, Journal of Political Economy, {95}, 1217-39.

Hart, S. (1989), Shapley Value, The New Palgrave: Game Theory, J. Eatwell, M. Milgate and P. Newman (Editors), Norton, pp. 210–216.

Knaster, B., (1946). \textit{Sur le Probl\`eme du Partage Pragmatique de H Steinhaus}, Annales de la Societ\'e Polonaise de Mathematique, {19}, 228-230. 

Mead, W., (1967). \textit{Natural Resource Disposal Policy - Oral Auction Versus Sealed Bids}, Natural Resources Journal, {7}, 194-224. 

Milgrom, P.R., (1987). \textit{Auction theory}, in T. Bewley (ed.), Advances in Economic Theory - Fifth World Congress 1985. London, Cambridge University Press, 1-32. 

Moulin, H., (1993). \textit{On the Fair and Coalitions-Strateyproof Allocation of Private Goods}, in Binmore, K.G. and Kirman, A.P. (eds), \textit{Frontiers of Game Theory}, Cambridge, MIT Press, 151-163. 

Shapley, L.S., (1953). \textit{ A Value for n-person Games}, in Kuhn, H.W., Tucker, A.W. (eds.) \textit{Contributions to the Theory of Games II} (Annals of Mathematics Studies 28), Princeton, Princeton University Press, 307-317. 

\end{verse}

\end{document}